\documentclass[a4,12pt, leqno]{amsart}
\usepackage{amsmath,amsfonts,amsthm,latexsym,amssymb,mathrsfs}
\usepackage{bbm}

\def\ind{\textrm{ind}}
\def\L{\mathcal{L}}
\def\X{\mathcal{X}}
\def\K{\mathcal{K}}
\def\C{\mathcal{C}}
\def\R{\mathcal{R}}
\def\H{\mathcal{H}}

\newtheorem{df}{Definition}[section]
\newtheorem{thm}[df]{Theorem}

\newtheorem{cor}[df]{Corollary}

\newtheorem{rema}[df] {Remark}

\begin{document}

\title[Analytically Riesz operators]{\bf Analytically Riesz operators and \\ Weyl and Browder type theorems}
\author[Enrico Boasso]{Enrico Boasso}


\maketitle

\begin{abstract}\noindent Several spectra of analytically Riesz operators will be characterized. These results
will led to prove  Weyl and Browder type theorems for the aforementioned class of operators. 
\end{abstract}

\section{Introduction}

\noindent Recall that given a Banach space $\X$ and an operator $T\in \L(\X)$, $T$ is said to be polynomially compact (respectively
polynomially Riesz), if there exists a complex polynomial $P$ such that $P(T)$ is a compact operator (respectively a Riesz operator).
The structure and the spectrum of a polynomially compact operators is well known; in fact, it was characterized in \cite{Gil}.
In addition, these results were extended to polynomially Riesz bounded and linear maps in \cite{Geo}.\par 

\indent More generally, an operator $T$ is said to be analytically Riesz, if  there exists an analytical function $f$
defined on a neighbourhood of the spectrum of $T$ such that $f(T)$ is Riesz. The structure and the spectrum of analytically
Riesz operators were studied in \cite{KS}. In particular, according to \cite[Theorem 1]{KS}, such an operator can be
decomposed as $T=T_0\oplus S$, where $f$ is locally zero at each point of the spectrum of $T_0$ but it is not locally zero  at any point of the
spectrum of $S$. In addition, $S$ either acts on a finite dimensional space 
or it has the decomposition $S=T_1\oplus\dots \oplus T_n$, and in this last case the spectra of $T_i$ are disjoint  and there are $\lambda_i$
with the property that $T_i-\lambda_i$ are Riesz and $f(\lambda_i)=0$ ($i=1,\dots ,n$).\par

\indent The first objective of this article is to characterize several spectra of analytically Riesz operators. These spectra are defined in the infinite
dimensional context and among others they are the Fredholm spectrum, the Weyl spectrum, the Browder spectrum, the upper Fredholm spectrum, the Weyl essential 
approximate point spectrum and the Browder essential approximate point spectrum. Moreover, the spectrum of $T$ will be fully described. This will be done in section 3.
To this end, however, some restrictions need to be considered.\par

\indent In fact, according to \cite[Theorem 1]{KS}, since $f$ is locally zero at each point of the spectrum of $T_0$, no information concerning the spectrum of $T_0$ can be obtained. 
Actually, note that $f(T)=f(T_0)\oplus f(S)$ and $f(T_0)=0$. What is more,
if $S$ is defined on a finite dimensional Banach space, then it has no sense to study the aforementioned spectra. As a result, the analytically Riesz operators 
 that will be considered in this article will be the following: $T$ will be a bounded and linear map
 defined on an infinite dimensional Banach space $\X$ and the holomorphic function $f$ with the property that $f(T)$ is Riesz will be assumed not to be locally zero at
any point of the spectrum of $T$. Naturally, polynomially Riesz and polynomially compact operators
belong to this subclass of analytically Riesz operators.\par

\indent In addition, the results obtained in section 3 will be applied  to prove  Weyl and Browder type theorems for
analytically Riesz operators. This will be done in section 4.\par

\section{Preliminary definitions and facts}

\indent Fron now on  $\X$ will denote an infinite dimensional complex Banach space and $\L(\X)$  the algebra of all bounded and linear maps
defined on and with values in $\X$. If $T\in \L(\X)$, then $N(T)$ and $R(T)$ will stand for the null space and the
range of $T$, respectively.  Note that $I\in \L (\X)$ will be the identity map defined on $\X$. In addition, $\K(\X)\subset \L (\X)$ will denote
the closed ideal of compact operators defined on $\X$, $\C (\X)$ the Calkin algebra of $\X$ and $\pi\colon \L (\X)\to\C (\X)$ the quotient map. Furthermore,
$\X^*$ will stand for the dual space of $\X$ and if $T\in \L (\X)$, then $T^*\in \L (\X^*)$ will denote the dual operator.  Moreover, given $T\in \L (\X)$, $\sigma(T)$ will
stand for the spectrum of $T$ and if  $K\subset \mathbb{C}$, then acc $K$ will be the set of limit points of $K$ and iso $K=K\setminus$ acc $K$  the set of isolated points of $K$. \par  

\indent Recall that $T\in L(\X)$  is said to be  a
\it Fredholm \rm operator if $\alpha(T)=\dim N(T)$ and $\beta(T)=\dim \X/R(T)$
are finite dimensional, in which case its \it index \rm is given by
$$
\ind(T)=\alpha(T)-\beta (T).
$$
If  $\alpha (T)$ and $\beta (T)$ are finite and equal, so that the index is zero,
$T$ is said to be a \it Weyl \rm operator. The set of Fredholm operators will be
dented by $\Phi (\X)$.\par

\indent The \it ascent \rm  (respectively \it descent\rm ) of
$T\in \L(\X)$ is the smallest non-negative integer $a$
(respectively $d$) such that $N(T^a)=N(T^{a+1})$ (respectively
$R(T^d)=R(T^{d+1})$); if such an integer does not exist, then
$asc(T)=\infty$ (respectively $dsc(T)=\infty$). The
operator $T$ will be said to be \it Browder\rm,
if it is Fredholm and its ascent and descent are finite.\par

These classes of operators  generate the
Fredholm or essential spectrum, the Weyl spectrum and the Browder spectrum,
which will be denoted by $\sigma_e(T)$, $\sigma_w(T)$ and $\sigma_b(T)$ respectively ($T\in\L (\X)$).
It is well known that
$$
\sigma_e(T)\subseteq \sigma_w(T)\subseteq \sigma_b(T)= \sigma_e(T)\cup \hbox{ acc }\sigma (T)\subseteq \sigma (T).
$$

\indent Recall that $T\in\L (\X)$ is said to be a \it Riesz operator\rm, if $\sigma_e(T)=\{0\}$.
The set of all Riesz operators defined on $\X$ will be denoted by $\R(\X)$. 
More generally, $T$ will be said to be an \it analytically Riesz operator\rm, if there exists a holomorphic function $f$ defined on an open 
neighbourhood of $\sigma (T)$  such that $f(T)\in \R (\X)$
($\H (\sigma(T))$ will denote the algebra of germs of analytic functions defined on open neighbourhoods of $\sigma (T)$).
In particular, $T$ will be said to be \it polynomially Riesz \rm (respectively  \it polynomially compact\rm),
 if there exists $P\in C[X]$ such that $P(T)\in \R(\X)$ (respectively $P(T)\in \K(\X)$). 

\indent The concept of Fredholm operator has been generalized. 
An operator $T\in \L(\X)$ will be said to be
 \it B-Fredholm\rm, if there exists
$n\in\mathbb N$ for which $R(T^n)$ is closed and the
induced operator $T_n\in \L(R(T^n))$ is Fredholm. Note that if for some $n\in\mathbb N$,
$T_n\in \L(R(T^n))$ is Fredholm, then $T_m\in \L(R(T^m))$ is
Fredholm for all $m\ge n$; moreover $\ind (T_n)=\ind (T_m)$, for
all $m\ge n$. Therefore, it makes sense to define the index of
$T$ by $\ind (T)=\ind (T_n)$. Recall that $T$ is said to be \it
 \it B-Weyl\rm, if $T$ is  B-Fredholm and
$\ind(T)=0$. Moreover, $T$ is said to be a \it B-Browder \rm operator, if
there is some $n\in \mathbb N$ such that $R(T^n)$ is closed, $T_n\in L(R(T^n))$
is Fredholm and $asc(T_n)$ and $dsc(T_n)$ are finite.
Naturally, from these classes of operators, the
B-Fredholm spectrum, the B-Weyl spectrum and the B-Browder spectrum of $T\in \L(\X)$ can be derived,
which will be denoted by $\sigma_ {BF}(T)$, $\sigma_ {BW}(T)$ and $\sigma_ {BB}(T)$,
respectively. Clearly,
$$
\sigma_ {BF}(T)\subseteq  \sigma_ {BW}(T)\subseteq \sigma_ {BB}(T)\subseteq \sigma (T).
$$

\indent On the other hand, recall that a Banach space operator $T\in \L(\X)$ is said to be Drazin invertible,
if there exists a necessarily unique $S\in \L(\X)$ and some $m\in \mathbb N$ such that
$$
T^m=T^mST, \hskip.3truecm STS=S, \hskip.3truecm ST=TS.
$$
Recall that necessary and sufficient for $T$ to be  Drazin invertible is that 
$asc (T)$ and $des (T)$ are finite (\cite[Theorem 4]{K}). If $DR(\L(\X))=\{ A\in \L(\X)\colon
A\hbox{ is Drazin invertible} \}$, then the Drazin spectrum of
$T\in \L(\X)$ is the set $\sigma_{DR}(T)=\{\lambda \in \mathbb{C}\colon
T-\lambda I\notin DR(\L(\X)) \}$, see \cite{BS, Bo3}. Note that according to \cite[Theorem 3.6]{Be} and \cite[Theorem 4]{K}, $\sigma_{BB}(T)=\sigma_{DR}(T)$. 
In particular,  $\sigma_ {BW}(T)\subseteq \sigma_ {DR}(T)\subseteq \sigma (T)$.
\par

\indent In order to state the (generalized) Weyl's and the (generalized) Browder's theorems, some
sets need to be recalled.
\par

\indent  Let $T\in \L(\X)$ and denote by $E(T)= \{\lambda\in
\hbox{iso}\hskip.1truecm \sigma(T)\colon 0<\alpha(A-\lambda)\}$ (respectively by
$E_0(T)= \{\lambda\in E(T)\colon \alpha(A-\lambda)<\infty\}$) the
set of eigenvalues of $T$ that are isolated in the spectrum of
$T$ (respectively the eigenvalues of finite multiplicity of $T$ 
isolated in $\sigma(T)$). In addition, denote  by $\Pi (T)=
\{\lambda\in \sigma (T)\colon
0<asc(T-\lambda)=dsc(T-\lambda)<\infty\}$ (respectively
$\Pi_0(T)=\{\lambda\in \Pi(T)\colon \alpha(T-\lambda)<\infty\}$)
the set of poles of $T$ (respectively the set of poles of finite rank of
$T$). Note that $\sigma (T)\setminus\sigma_{DR}(T)=\Pi(T)$ (\cite[Theorem 4]{K})
and $\sigma (T)\setminus\sigma_b(T)=\Pi_0 (T)$ (\cite[Proposition 2]{Ba}).
In particular, $\sigma_{DR} (T)\subseteq \sigma_b (T)$.\par

\begin{df}Consider a Banach space $\X$ and $T\in \L(\X)$. Then, it will be said that\par
\rm (i)\it Weyl's theorem ($Wt$) holds for $T$, if $\sigma_w(T)=\sigma(T)\setminus E_0(T)$.\par
\rm (ii) \it Generalized Weyl's theorem ($gWt$) holds for $T$, if $\sigma_{BW}(T)=\sigma(T)\setminus E(T)$.\par
\rm (iii) \it Browder's theorem ($Bt$) holds for $T$, if $\sigma_w(T)=\sigma(T)\setminus \Pi_0(T)$.\par
\rm (iv) \it Generalized Browder's theorem ($gBt$) holds for $T$, if $\sigma_{BW}(T)=\sigma(T)\setminus \Pi(T)$.
\end{df}

\indent Next the definitions of (generalized) $a$-Weyl and (generalized) $a$-Browder theorems will be recalled.
To this end, however, some preparation is needed first.\par

\indent  Recall that $T\in \L(\X)$ is said to be
\it bounded below\rm, if $N(T)=0$ and $R(T)$ is closed. Denote
the \it approximate point spectrum \rm of $T$ by
$\sigma_a(T)=\{\lambda\in \mathbb C \colon A-\lambda I \hbox{ is
not bounded below} \}$. In addition, if $R(T)$ is closed and $\alpha (T)$ is finite,
then $T\in  \L(\X)$ is said to be  \it upper semi-Fredholm\rm. Note that in this case
ind $(T)$ is well defined. Let $\sigma_{\Phi_+}(T)=\{\lambda\in \mathbb C \colon A-\lambda I \hbox{ is
not upper semi-Fredholm\rm } \}$ denote the upper Fredholm spectrum.\par
\indent The \it Weyl essential approximate point spectrum
\rm of $T\in \L(\X)$ is the set $\sigma_{aw}(T)=\{\lambda\in \sigma_a(T)\colon
T-\lambda I\hbox{ is not upper semi-Fredholm or } 0<\ind (A-\lambda I) \}$ (\cite{R2}). In addition, the \it Browder essential  approximate point
spectrum \rm of $T\in \L(\X)$ is the set
$\sigma_{ab}(T) =\{\lambda\in \sigma_a(T)\colon \lambda\in \sigma_{aw}(T)
 \hbox{ or } asc(T-\lambda I)=\infty\}$ 
 (\cite{R2}). It is clear that
$\sigma_{aw}(A)\subseteq \sigma_{ab}(A)\subseteq \sigma_a(A)$.\par
  
\indent On the other hand, upper semi B-Fredholm operators can be defined
in a similar way as B-Fredholm operators. Set $\sigma_ {SBF_+^-}(T)= \{\lambda\in
\mathbb C\colon T-\lambda I  \hbox{ is not upper semi B-Fredholm or
} 0<\ind (A-\lambda)\}$ (\cite{B6}).\par 

Next denote by $LD(\X)= \{ T\in \L(\X)\colon \hbox{  }a=
asc(T)<\infty\hbox{  and } R(T^{a+1 }) \hbox{ is}$
$\hbox{closed}\}$ the
set of \it left Drazin invertible \rm operators. Then, given
$T\in B(\X)$, the \it left Drazin spectrum \rm of $T$ is the set
$\sigma_{LD}(T)= \{\lambda\in \mathbb C\colon T-\lambda I\notin
LD(\X)\}$. Note that according to \cite[Lemma 2.12]{B6}, $\sigma_
{SBF_+^-}(T)\subseteq \sigma_{LD}(T)\subseteq \sigma_a(T)$.\par

\indent Let $T\in \L (\X)$ and denote by  $E^a(T)= \{\lambda\in
\hbox{iso}\hskip.1truecm \sigma_a(T)\colon 0<\alpha(T-\lambda I)\}$ (respectively by
$E_0^a(T)= \{\lambda\in E^a(T)\colon \alpha(T-\lambda T)<\infty\}$)
the set of eigenvalues of $T$ that are isolated in $\sigma_a(T)$
(respectively, the eigenvalues of finite multiplicity of $T$ isolated in
$\sigma_a(T)$). In addition, denote by $\Pi ^a(T)= \{\lambda\in
\sigma_a(T)\colon a=asc(T-\lambda I)<\infty \hbox{ and
}R(T-\lambda I)^{a+1} \hbox{ is closed}\}$ (respectively
$\Pi_0^a(T)= \{\lambda\in\Pi ^a(T)\colon
\alpha(T-\lambda I)<\infty\}$) the set of left poles of $T$
(respectively, the left poles of finite rank of $T$).
Clearly, $\sigma_a (T)\setminus\sigma_{LD}(T)=\Pi^a(T)$. Moreover, according to \cite[Corollary 2.2]{R2},
\cite[Corollary 1.3.3]{CPY} and \cite[Corollary 1.3.4]{CPY}, $\sigma_a (T)\setminus\sigma_{ab}(T)=\Pi^a_0(T)$. \par

\indent Next follows the definitions of (generalized) $a$-Weyl and (generalized) $a$-Browder theorems.\par

\begin{df}Consider a Banach space $\X$ and $T\in \L(\X)$. Then, it will be said that\par
\rm (i) \it $a$-Weyl's theorem ($a$-$Wt$) holds for $T$, if $\sigma_{aw}(T)=\sigma_a(T)\setminus E^a_0(T)$.\par
\rm (ii)\it Generalized $a$-Weyl's theorem ($a$-$gWt$) holds for $T$, if $\sigma_{SBF_+^-}(T)=\sigma_a(T)\setminus E^a(T)$.\par
\rm (iii) \it $a$-Browder's theorem ($a$-$Bt$) holds for $T$, if $\sigma_{aw}(T)=\sigma_a(T)\setminus \Pi_0^a(T)$.\par
\rm (iv)\it Generalized $a$-Browder's theorem ($a$-$gBt$) holds for $T$, if $\sigma_{SBF_+^-}(T)=\sigma_a(T)\setminus \Pi^a(T)$.

\end{df}

\indent Finally, recall that an operator $T\in L(\X)$ is said to
have the single--valued extension property (SVEP for short),  at a (complex)
point $\lambda_0$, if for every open disc ${\mathcal D}$ centered
at $\lambda_0$ the only analytic function $f:{\mathcal
D}\longrightarrow \X$ satisfying $(T-\lambda)f(\lambda)=0$ is the
function $f\equiv 0$. It will be said that $T$ has SVEP, if it has SVEP at every point
of $\mathbb C$. Trivially, every operator $T$ has SVEP at
points of the resolvent $\rho(A)={\C}\setminus \sigma(T)$. Also
$T$ has  SVEP at every point of the boundary $\partial \sigma(T)$ of the spectrum.
See \cite[Chapter 2]{A} for
more information on operators with SVEP.\par

\section{Spectra of analytically Riesz operators}

As it was explained in the Introduction, to study several spectra of analytically Riesz operators, it is necessary to consider a subclass of
these  operatores. To this end,  given  $T\in \L (\X)$,
set  $\H_{\R NLZ}(\sigma (T))=\{f\in\H(\sigma (T))\colon f(T)\in \R (\X) \hbox{ and } f \hbox{ is not locally zero at any point of }
\sigma (T)\}$. In the following remark some basic facts will be considered.\par

\begin{rema}\label{rem28}\rm Let $\X$ be an infinte dimensional Banach space and consider $T\in \L(\X)$ such that 
$\H_{\R NLZ}(\sigma (T))\neq\emptyset$. Let $f\in \H_{\R NLZ}(\sigma (T))$.\par
\noindent (i). According to the proof of \cite[Thorem 1]{KS}, 
$$
\sigma_e(T)\subseteq \sigma_w (T)\subseteq\sigma_b( T)\subseteq  f^{-1}(0).
$$
\noindent (ii). Recall that $\hbox{acc }\sigma (T)\subseteq \sigma_{DR} (T)$ (\cite[Theorem 12(iv)]{Bo3}).
In particular, since $\sigma_{DR} (T)\subseteq \sigma_b (T)$,
$$
\hbox{acc }\sigma (T)\subseteq \sigma_{DR} (T)\subseteq f^{-1}(0).
$$
\noindent (iii). Since $f$ is not locally zero at any point of the spectrum of $T$, the set 
 $f^{-1}_{\sigma (T)}(0)=\{\lambda\in\sigma (T)\colon f(\lambda)=0\}$ is finite. Define
$$n_T= \hbox{min} \{\sharp f^{-1}_{\sigma (T)}(0)\colon f\in \H_{\R NLZ}(\sigma (T))\}.$$ 
The function $h\in \H_{\R NLZ}(\sigma (T))$ will be said to be \it a minimal analytical function associated to $T$, \rm
if   $n_T=\sharp h^{-1}_{\sigma (T)}(0)$.\par
\noindent In addition, recall that there exists an analytic function $g$ defined on a neighbourhod of $\sigma (T)$ such that $g(z)\neq 0$, $z\in \sigma (T)$, and
$f(z)=(z-\lambda_1)^{k_1}\dots (z-\lambda_n)^{k_n}g(z)$, where  $f^{-1}_{\sigma (T)}(0)=\{\lambda_1,\dots ,\lambda_n\}$.\par
\noindent On the other hand, when $T\in \L(\X)$ is a polynomially Riesz operator (respectively a polynomially compact operator), 
a polynomial $Q\in\mathbb{C}[X]$ will be said to a \it minimal
polynomial associated to $T$, \rm if $Q(T)\in\R(\X)$ (respectively $Q(T)\in\K(\X)$), and if $P\in\mathbb{C}[X]$ is such that $P(T)\in R(\X)$ (respectively $P(T)\in\K(\X)$), then the degree of $Q$
is less or equal to the degree of $P$.\par
\noindent (iv). Note that $\sigma (T)$ is countable. In fact, it is a consequence of (ii)-(iii) and  \cite[Theorem 2.2]{Boasso}.
\end{rema}

\indent Applying Remark \ref{rem28}, the following properties can be deduced.\par

\begin{cor}\label{cor4}Let $\X$ be a Banach space and consider $T\in \L (\X)$ such that $\H_{\R NLZ}(\sigma (T))\neq\emptyset$.
Then, the following statements hold.\par
\noindent \rm (i) \it $T$ and $T^*$ have the SVEP.\par
\noindent \rm (ii) \it $\sigma (T)=\sigma_a(T)$ and $\sigma_e (T)=\sigma_{\Phi_+}(T)$.\par
\noindent \rm (iii) \it $E(T)=E^a(T)$ and $E_0(T)=E_0^a(T)$.\par
\noindent \rm (iv) \it $\Pi (T)=\Pi^a(T)$ and $\Pi_0(T)=\Pi_0^a(T)$.
\end{cor}
\begin{proof} (i). According to Remark \ref{rem28}(iv), $\sigma (T)=\partial \sigma (T)$.
In addition, since  $\sigma (T)=\sigma (T^*)$, $\sigma (T^*)=\partial \sigma (T^*)$.
In particular, $T$ and $T^*$ have the SVEP.\par
\noindent (ii). Apply \cite[Corollary 2.45, Chapter 2, Section 3]{A} and \cite[Corollary 3.53, Chapter 3, Section 4]{A}, respectively.\par
\noindent (iii). It can be derived from statement (ii).\par
\noindent (iv). According to \cite[Theorem 2.7]{BBO} and \cite[Theorem 12(i)]{Bo3},
$\Pi (T)=\Pi^a(T)$, which in turn implies that $\Pi_0(T)=\Pi_0^a(T)$.
\end{proof}

\indent In the following theorem the Fredholm, Weyl and Browder spectra of analytically Riesz operators will be
characterized.\par

\begin{thm}\label{thm2}Let $\X$ be a Banach space and consider $T\in \L (\X)$ such that $\H_{\R NLZ}(\sigma (T))\neq\emptyset$.
Let $h\in \H_{\R NLZ} (\sigma (T))$  be a minimal analytic function associated to $T$. Then, the following statements hold.\par
\noindent \rm (i) \it $\sigma_e (T)= \sigma_w (T)=\sigma_b( T)=h^{-1}_{\sigma (T)}(0)$.\par
\noindent \rm (ii) \it $\sigma (T)\setminus h^{-1}_{\sigma (T)}(0)=\Pi_0 (T)$.\par
\end{thm}
\begin{proof}(i). According to Remark \ref{rem28}(i) (\cite[Theorem 1]{KS}), it is enough to prove that $h^{-1}_{\sigma (T)}(0)\subseteq \sigma_e(T)$. 
Recall that there exist $n\in\mathbb{N}$ and $\lambda_1,\dots  , \lambda_n\in\sigma (T)$ such that 
$h^{-1}_{\sigma (T)}(0)=\{\lambda_1,\dots  , \lambda_n\}$ and $h(z)=(z-\lambda_1)^{k_1}\dots (z-\lambda_n)^{k_n}g(z)$,
where $g$ is an analytic function defined on a neighbourhod of $\sigma (T)$ such that $g(z)\neq 0$, $z\in \sigma (T)$.
Suppose that there is $j$, $1\le j\le n$, such that $\lambda_j\notin\sigma_e (T)$. In particular,
$(z-\lambda_j)^{k_j}\in \L (\X)$ is a Fredholm operator.\par

\indent Let $U_i\in \L (\X)$ and $K_i\in \K (\X)$, $i=1, 2$, such that $U_1 (T-\lambda_j)^{k_j}=I-K_1$ and $ (T-\lambda_j)^{k_j}U_2=I-K_2$.
Set $h_1(z)=\Pi_{1\le i\le n, i\neq j}(z-\lambda_i)^{k_i}g(z)$ and consider $\Pi \colon \L (\X)\to\C (\X)$. Note that
$\sigma (\pi (h(T)))=\{ 0\}$, $\pi (U_1)=\pi (U_2)= (\pi(T-\lambda_j)^{k_j})^{-1}$ and
$$
\pi (h(T))=\pi (h_1(T))\pi(T-\lambda_j)^{k_j}=\pi(T-\lambda_j)^{k_j}\pi (h_1(T)).
$$
\indent However, an easy calculaton proves that 
$$
\pi (U_1) \pi (h(T))=\pi (h(T))\pi (U_1)=\pi (h_1(T)).
$$
\noindent In particular, $\sigma (\pi (h_1(T))=\{ 0\}$, equivalenty $h_1(T)\in \R (\X)$, which is impossible
because $n=n_T$.\par
\noindent (ii). Apply statement (i) and \cite[Proposition 2]{Ba}.
\end{proof}

\begin{rema}\label{rema3}\rm  Let $\X$ be a Banach space and  $T\in \L (\X)$. Suppose that $\H_{\R NLZ}(\sigma (T))\neq\emptyset$
an let $h\in \H_{\R NLZ} (\sigma (T))$  be a minimal analytic function associated to $T$.\par
\noindent (i). Recall that $\Pi (T)\subseteq $ iso $\sigma (T)$. 
Let $I(T)=$ iso $\sigma (T)\setminus\Pi (T)$. Then, according to \cite[Theorem 12]{Bo3}, $\sigma (T)=\sigma_{DR} (T)\cup \Pi (T)$, 
$\sigma_{DR} (T)\cap \Pi (T)=\emptyset$ and  $\sigma_{DR}(T)=\hbox{ \rm acc }\sigma (T)\cup I(T)$. Then, since $\sigma_{DR}(T)\subseteq \sigma_b (T)$,
according to Theorem  \ref{thm2},
$$
h^{-1}_{\sigma (T)}(0)= \sigma_e (T)= \sigma_b( T)=\hbox{ acc } \sigma (T)\cup I(T)\cup (\Pi (T)\setminus \Pi_0 (T)).
$$
\noindent (ii). Suppose that $T\in \L (\X)$ is polynomially Riesz (respectively polynomially compact) and consider $Q\in C[X]$ a minimal polynomial associated to
$T$. Applying arguments similar to the ones in Remark \ref{rem28}(i) and Theorem  \ref{thm2} it is not difficult to pove that
$$
\sigma_e (T)= \sigma_w (T)=\sigma_b( T)=\{\lambda\in\mathbb{C}\colon Q(\lambda)=0\}.
$$
Therefore, when $T$ is polynomially Riesz (respectively polynomially compact) and $Q$ is a minimal polynomial associated to
$T$, $n_T=\sharp \{\lambda\in\mathbb{C}\colon Q(\lambda)=0\}$.
\end{rema}

\indent Next the Weyl approximation essential point spectrum and the Browder approximation
essential point spectrum will be characterized.\par

\begin{thm}\label{cor44}Let $\X$ be a Banach space and consider $T\in \L (\X)$ such that $\H_{\R NLZ}(\sigma (T))\neq\emptyset$.
Let $h\in \H_{\R NLZ} (\sigma (T))$  be a minimal analytic function associated to $T$.Then, the following statements hold.\par
\noindent \rm (i) \it $\sigma (T)\setminus h^{-1}_{\sigma (T)}(0)=\Pi_0^a(T)$.\par
\noindent \rm (ii) \it  $\sigma_{\Phi_+} (T)= \sigma_{aw} (T)=\sigma_{ab}( T)=h^{-1}_{\sigma (T)}(0)$.
\end{thm}
\begin{proof}According to Corollary \ref{cor4}(iv) and Theorem \ref{thm2}(ii), statement (i) holds.\par
\indent On the other hand, since $\Pi_0^a(T)=\Pi_0(T)$, according to Theorem \ref{thm2}(i) and Corollary \ref{cor4}(ii),
$$
\sigma_{\Phi_+} (T)\subseteq \sigma_{aw} (T)\subseteq \sigma_{ab}(T)=\sigma_b (T)=\sigma_e (T)=\sigma_{\Phi_+} (T).
$$ 
\end{proof}

\section{Weyl and Browder type theorems}

\indent In first place, Browder type theorems will be considered.\par

\begin{thm}\label{thm4.1}Let $\X$ be a Banach space and consider $T\in \L (\X)$ such that $\H_{\R NLZ}(\sigma (T))\neq\emptyset$.
Then, the following statements hold.\par
\noindent \rm (i) \it (Generalized) Browder's theorem holds for $T$.\par
\noindent \rm (ii) \it (Generalized) $a$-Browder's theorem holds for $T$.\par
\noindent \rm (iii) \it Given $f\in \H (\sigma (T))$, generalized $a$-Browder's theorem holds for $f(T)$ and $f(T^*)$.
\end{thm} 
\begin{proof} (i). Theorem \ref{thm2}(i) implies that Browder's theorem holds. According to \cite[Theorem 2.1]{AZ}, Browder's theorem and generalized Browder's theorem are
equivalent. \par
\noindent (ii). Theorem \ref{cor44}(ii) implies that $a$-Browder's theorem holds. According to \cite[Theorem 2.2]{AZ}, $a$-Browder's theorem and generalized $a$-Browder's theorem are
equivalent. \par
\noindent  (iii). Apply Corollary \ref{cor4}(i) and \cite[Theorem 3.2]{Am}.
\end{proof} 

\indent Recall that generalized $a$-Browder's theorem is equivalent to $a$-Browder's theorem (\cite[Theorem 2.2]{AZ}) and it implies
generalized Browder's theorem (\cite[Theorem 3.8]{B6}), which in turn is equivalent to Browder's theorem (\cite[Theorem 2.1]{AZ}). Therefore, under the same hypothesis of
Theorem \ref{thm4.1}(iii),   (generalized) Browder's theorem and $a$-Browder's theorem hold for $f(T)$ and $f(T^*)$.\par

\begin{cor}\label{cor4.2}Let $\X$ be a Banach space and consider $T\in \L (\X)$ such that $\H_{\R NLZ}(\sigma (T))\neq\emptyset$.
Then,
$$
\sigma_{BW} (T)=\sigma_{DR} (T)=\sigma_{LD} (T)=\sigma_{SBF_+^-}(T).
$$
\end{cor}
\begin{proof}Note that generalized Browder's theorem (respectively $a$-generalized Browder's theorem)
is equivalent to $\sigma_{BW} (T)=\sigma_{DR} (T)$ (respectively $\sigma_{LD} (T)$
$=\sigma_{SBF_+^-}(T)$).
In adition, since $\sigma (T)=\sigma_a(T)$ and $\sigma (T)=\partial (T)$ (Corollary \ref{cor4}),  according to \cite[Theorem 2.7]{BBO}, $\sigma_{DR} (T)=\sigma_{LD} (T)$.
\end{proof}

\indent Next Weyl type theorems will be considered.\par

\begin{thm}\label{thm3}Let $\X$ be a Banach space and consider $T\in \L (\X)$ such that $\H_{\R NLZ}(\sigma (T))\neq\emptyset$.
Then, the following statements are equivalent.\par
\noindent \rm (i) \it Weyl's theorem holds for $T$ (respectively $T^*$).\par
\noindent \rm (ii) \it $a$-Weyl's theorem holds for $T$ (respectively $T^*$).\par
\noindent In addition, the following statements are equivalent.\par
\noindent \rm (iii) \it Generalized Weyl's theorem holds for $T$ (respectively $T^*$).\par
\noindent \rm (iv) \it Generalized $a$-Weyl's theorem holds for $T$ (respectively $T^*$).\par
\end{thm}
\begin{proof}Recall that $T$ and $T^*$ have the SVEP (Corollary \ref{cor4}(i)).
Then apply \cite[Theorem 3.6]{A0} and \cite[Theorem 3.1]{Am}
\end{proof}

\begin{thm}\label{thm3}Let $\X$ be a Banach space and consider $T\in \L (\X)$ such that $\H_{\R NLZ}(\sigma (T))\neq\emptyset$.
Suppose that \rm iso \it $\sigma (T)=E (T)=\Pi (T)$. Then, given $f\in\H (\sigma (T))$,
generalized $a$-Weyl's theorem holds for $f(T)$.
\end{thm}
\begin{proof}Apply Corollary \ref{cor4}(i) and \cite[Corollary 2.4]{AZ}
\end{proof}

\indent Since  generalized $a$-Weyl's theorem implies $a$-Weyl's theorem (\cite[Theorem 3.11]{B6}),
generalized Weyl's theorem (\cite[Theorem 3.7]{B6}) and Weyl's theorem (\cite[Corollary 3.10]{B6}), the statement of Theorem \ref{thm3}
holds if instead of generalized $a$-Weyl's theorem the other Weyl type theorems are considered.

\bibliographystyle{amsplain}

\begin{thebibliography}{99}


\bibitem{A} P. Aiena, Fredholm and Local Spectral Theory
with Applications to Multipliers, Kluwer, 2004.

\bibitem{A0} P. Aiena, Classes of operators satisfying $a$-Weyl's theorem,
Studia Math. 169 (2005), 105-122.

\bibitem{Am}M. Amouch, Generalized $a$-Weyl's theorem and the single-valued extension
property, Extracta Math. 21 (2006), 51-65.

\bibitem{AZ} M. Amouch and H. Zguitti, On the equivalence of
Browder's and generalized Browder's theorem, Glasgow Math. J.
48 (2006), 179-185.

\bibitem{Ba} B. A. Barnes, Riesz points and Weyl's theorem, Integral Equations Operator Theory 34 (1999), 187-196. 

\bibitem{BBO} O. Bel Hadj Fredj, M. Burgos and M. Oudghiri, Ascent spectrum and essential ascent spectrum,
Studia Math. 187 (2008), 59-73.

\bibitem{Be} M. Berkani, Restriction of an operator to the range of its powers, Studia Math. 140 (2000), 163-175.

\bibitem{B6} M. Berkani and J. J. Koliha,  Weyl type theorems for bounded linear operators,
Acta Sci. Math (Szeged) 69 (2003), 359-376.

\bibitem{BS} M. Berkani and M. Sarih, An Atkinson-type theorem for B-Fredholm operators, 
Studia Math. 148 (2001), 251-257.

\bibitem{Bo3} E. Boasso, Drazin spectra of Banach space
operators and Banach algebra elements, J. Math. Anal. Appl. 359 
(2009), 48-55.

\bibitem {Boasso} E. Boasso, The Drazin spectrum in Banach
algebras, An Operator Theory Summer: Timisoara June 29-July 4 2010, 
International Book Series of Mathematical Texts, Theta Foundation, Bucharest, 2012, 21-28

\bibitem{CPY} S. R. Caradus, W. E. Pfaffenberg and B. Yood, Calkin algebras and algebra
of operators on Banach space, Marcel Dekker, Inc., New York, 1974.

\bibitem{Geo} C. Gheorghe, Some remarks on structure of polynomially Riesz operators. Atti Accad. Naz. Lincei Rend. Cl. Sci. Fis. Mat. Natur.  (8) 54 (1973), 42-45.

\bibitem{Gil} F. Gilfeather, The structure and asymptotic behaviour of polynomially compact operators,
Proc. Amer. Math. Soc. 25 (1970), 127-134.

\bibitem{KS} M. A. Kaashoek and M. R. F. Smyth, On operators $T$ such that $f(T)$ is Riesz or meromorphic,
Proc. Roy. Irish Acad. Sect. A 72 (1972), 81-87.

\bibitem{K} C. King, A note on Drazin inverses, Pacific J. Math. 70  (1977), 383-390.
  
\bibitem{R2} V. Rakocevic,  Approximate point spectrum and commuting compact perturbations,
Glasgow Math. J. 28 (1986), 193-198. 

\end{thebibliography}

\vskip.3truecm
\noindent Enrico Boasso\par
\noindent E-mail address: enrico\_odisseo@yahoo.it
\end{document}